\documentclass[english,twoside]{article}
\usepackage{cite}
\usepackage{amsmath,amssymb}
\usepackage{latexsym}
\usepackage[english]{babel}
\usepackage[latin1]{inputenc}
\usepackage{amsthm}
\usepackage{graphicx}

\newcommand{\N}{\mathbb{N}}
\newcommand{\Z}{\mathbb{Z}}

\newcommand{\R}{\mathbb{R}}

\newcommand{\A}{\mathbb{A}}

\newcommand{\s}{\mathbb{S}}

\newcommand{\tore}{\mathbb{T}}
\newcommand{\T}[1]{\tore^{#1}}

\newcommand{\fonc}[3]{#1:#2\to#3}

\DeclareMathOperator{\ord}{ord}
\DeclareMathOperator{\inv}{Inv}
\DeclareMathOperator{\leb}{Leb}
\DeclareMathOperator{\fix}{Fix}
\DeclareMathOperator{\lrho}{\rho_{loc,z}}
\DeclareMathOperator{\homeo}{Homeo}
\DeclareMathOperator{\diff}{Diff}
\newcommand{\homeog}[3]{\homeo^{#1}_{#2}\left(#3\right)}
\theoremstyle{theorem}
\newtheorem{lema}{Lemma}[section]
\newtheorem{coro}[lema]{Corollary}
\newtheorem{prop}[lema]{Proposition}

\newtheorem*{theoa}{Theorem A}

\newtheorem*{theob}{Theorem B}

\newtheorem*{theoc}{Theorem C}

\newtheorem*{theod}{Theorem D}

\newtheorem*{theoe}{Theorem E}

\theoremstyle{remark}

\title{On periodic groups of homeomorphisms of the 2-dimensional sphere}
\author{\textsc{Jonathan Conejeros}}
\date{}

\begin{document}
\maketitle

\begin{abstract}
We prove that every finitely-generated group of homeomorphisms of the 2-dimensional sphere all of whose elements have a finite order
which is a power of 2 and so that there exists a uniform bound for the order of group elements is finite. We prove a similar result for groups of
area-preserving homeomorphisms without the hypothesis that the orders of group elements are powers of 2 provided there is an element of even order.

\end{abstract}

\section{Introduction}

Despite some remaining open questions, there is a very complete understanding of group actions on 1-manifolds (see \cite{ghys} and \cite{navas}).
However, when passing to the 2-dimensional setting,
many natural and fundamental questions remain unsolved. One of the most striking ones is related to the Burnside problem.

Recall that Burnside (see \cite{bu}) proved that every finitely-generated linear group all of whose elements have finite order and such that there exists a uniform bound for the order of group elements is actually finite. This result has been extended to some
other contexts, but fails in general, as it is shown by classical examples due to Golod (see \cite{golod}). Later, Ol'shanskii (see \cite{ol}), Ivanov (see \cite{iv}), and Lysenok (see \cite{ly}) exhibited many other examples of infinite, finite-generated groups all of whose elements have a finite order which is bounded by a uniform constant. The case of groups of homeomorphisms is particularly interesting. The following question seems to be folklore: Does there exists an infinite, finitely-generated group of homeomorphisms of the 2-dimensional sphere all of whose elements have finite order?
Some progress on this question has been made by Guelman and Liousse (provided there is a finite orbit for the action and -in some cases- that all maps
involved are of class $C^1$), and Hurtado (provided the action is by $C^{\infty}$ volume-preserving diffeomorphisms and there is a uniform bound for the order of group elements).
The main result of this paper yields a new positive result for actions by homeomorphisms under some hypothesis on the orders of group elements. For short, in what follows, we will call {\em periodic} a group in which all elements have finite order, we will say that such a group is a {\em 2-group} if the orders of group elements are powers of 2. Also, we will say
that a periodic group has {\em uniformly bounded order} if there exists a uniform bound for the order of group elements.

\begin{theoa}
  Let $G$ be a finitely-generated $2$-group of homeomorphisms of the $2$-dimensional sphere. Suppose that
  $G$ has uniformly bounded order. Then $G$ is finite.
\end{theoa}

We note that the composition of two orientation-reversing homeomorphisms is orientation preserving. We deduce that the subgroup of orientation-preserving homeomorphisms has at most index 2 in the group $G$ above. Moreover, Schreier's lemma states that any finite-index subgroup in a finitely-generated group is finitely generated. Hence, in order to prove Theorem A, it is enough to show that a finitely-generated $2$-group of {\em orientation-preserving} homeomorphisms of the $2$-dimensional sphere is finite provided there is a uniform bound for the orders of the group elements. As a first step to prove this, we will show the next result which is interesting by itself.

\begin{theob}
  Let $G$ be a finitely-generated $2$-group of orientation-preserving homeomorphisms of the $2$-dimensional sphere. Suppose that $G$ acts with a global fixed point. Then $G$ is finite and cyclic. 
\end{theob}

The second step in the proof is the following.
\begin{theoc}
Let $G$ be a finitely-generated $2$-group of orientation-preserving homeomorphisms of the $2$-dimensional sphere.
Suppose that $G$ has a finite orbit of cardinality $2$.  Then $G$ is finite. More precisely, it is either a cyclic or a dihedral group.
\end{theoc}

As a by product of our methods, we obtain the following result for groups of area-preserving homeomorphisms.

\begin{theod}
  Let $G$ be a finitely-generated periodic group of area-preserving homeomorphisms of the $2$-dimensional sphere.
  Suppose that $G$ has uniformly bounded order and contains an element of even order. Then $G$ is finite.
\end{theod}

As above, in order to prove Theorem D, it is enough to show an analog of Theorem C in this setting.

\begin{theoe}
  Let $G$ be a finitely-generated periodic group of area-preserving homeomorphisms of the $2$-dimensional sphere. Suppose that $G$ has a finite orbit of cardinality $2$.
  Then $G$ is finite. More precisely, it is either a cyclic or a dihedral group.
\end{theoe}

\noindent \textbf{Acknowledgments:} The preparation of this article was funded by the
FONDECYT postdoctoral grant N° 3170455 entitled ``Algunos problemas para grupos de homeomorfismos de superficies''. I thank A. Navas for several useful discussions, suggestions, and corrections.

%
%

\section{Preliminary results}

\subsection{Local rotation set}

Since local dynamics (more precisely, the dynamics around a fixed point) does not fit into a compact framework, we consider only rotation numbers of ``good orbits''. This means that, in order to get a definition of a rotation set which is invariant under conjugacy, we consider only recurrent points close to the fixed point.\\

Let $\overline{g}$ be a homeomorphism of the plane $\R^2$ that preserves the orientation and fixes the vector $\mathbf{0}:=(0,0)\in \R^2$. We will denote $\widetilde{\A}= \R\times (0,+\infty)$ the universal covering of $\R^2\setminus \{\mathbf{0}\}$. Let $\fonc{\widetilde{\pi}}{\widetilde{\A}}{\R^2\setminus\{\mathbf{0}\}}$ be the corresponding universal covering map, and $\fonc{p_1}{\R\times (0,+\infty)}{\R}$ the projection on the first coordinate. Let $\widetilde{\overline{g}}$ be a lift of $\overline{g}$ to $\widetilde{\A}$. Following \cite{leroux}, we say that the {\em rotation number (around $\mathbf{0}$)} of a $\overline{g}$-recurrent point $x\in \R^2\setminus \{\mathbf{0}\}$ under $\widetilde{\overline{g}}$ is well-defined and equal to $\rho_{\mathbf{0}}(\widetilde{\overline{g}},x)\in \R\cup \{-\infty\}\cup \{+\infty\}$ if for every sequence of integers $(n_k)_{k\in \N}$ which goes to $+\infty$ such that $ (\overline{g}^{n_k}(x))_{k\in \N}$ converges to $x$,  the sequence $(\rho_{n_k}(\widetilde{\overline{g}},x))_{k\in \N}$ defined as
$$ \rho_{n_k}(\widetilde{\overline{g}},x):=\frac{1}{n_k}(p_1(\widetilde{\overline{g}}^{n_k}(\widetilde{x}))-p_1(\widetilde{x})),$$
where $\widetilde{x}$ is a point in $\widetilde{\pi}^{-1}(x)$, converges to  $\rho_{\mathbf{0}}(\widetilde{\overline{g}},x)$. Notice that this definition does not depend on the choice of $\widetilde{x}\in \widetilde{\pi}^{-1}(x)$.

The \textit{local rotation set (around the fixed point $\mathbf{0}$) of $\widetilde{\overline{g}}$}, which we denote as $\rho_{\mathbf{0}}(\widetilde{\overline{g}})$,  is the set of all rotation numbers of recurrent points of $\overline{g}$.

We have the following properties (see \cite{leroux} for more details):
\begin{enumerate}
  \item The rotation number of a recurrent point and, consequently, the local rotation set, are invariant under (local) oriented topological conjugacy. More precisely, if $\varphi$ is a homeomorphism of $\R^2$ that preserves the orientation and fixes $\mathbf{0}\in \R^2$ and $\widetilde{\varphi}$ is a lift of $\varphi$ to $\widetilde{\A}$,
  then $$    \rho_{\mathbf{0}}(\widetilde{\varphi}^{-1} \widetilde{\overline{g}} \widetilde{\varphi}) =   \rho_{\mathbf{0}}(\widetilde{\overline{g}}).$$
    \item For every $p,q\in\Z$, we have  $  \rho_{\mathbf{0}}(\widetilde{\overline{g}}^q +(p,0))= q \rho_{\mathbf{0}}(\widetilde{\overline{g}})+p $. A similar formula holds for the rotation number of a recurrent point.
\end{enumerate}

\subsection{Periodic, orientation-preserving homeomorphisms of the 2-dimensional sphere}

We say that an orientation-preserving homeomorphism $g$ of the 2-dimensional sphere is {\em periodic} if its order is finite, that is, if there exists an integer $q$ such that $g^q=Id$. We recall that Ker\'ej\'art\'o proved that every periodic, orientation-preserving homeomorphism of the $2$-dimensional sphere is conjugate to a rotation (see \cite{kere}  and \cite{coko} ). Formally, we have the following proposition.

\begin{prop}\label{kerejartolemma}
  Let $g$ be a periodic, orientation-preserving homeomorphism of the 2-dimensional sphere. Then there
  exist an orientation-preserving homeomorphism $h$ of the 2-dimensional sphere and a rotation
  $R$ such that $hgh^{-1}=R$. In particular, if $g$ is nontrivial, then it has exactly two fixed points,
  and every point that is not fixed is periodic. Besides, the periods of all nonfixed, periodic points are the same.
\end{prop}

%
%
%
%

\subsection{Local rotation set for periodic homeomorphisms}

In our setting, let $g$ be a periodic homeomorphism that preserves the orientation of $\s^2$ and fixes a point $z\in \s^2$. Considering a chart $\phi$ centered at $z$, we have that $\overline{g}= \phi g\phi^{-1}$ is a periodic homeomorphism of the plane $\R^2$ that preserves the orientation and fixes the vector $\phi(z)=\mathbf{0}\in \R^2$. Let $\widetilde{\overline{g}}$ be a lift of $\overline{g}$ to $\widetilde{\A}$ the universal covering to $\R^2\setminus \{\mathbf{0}\}$, which, as before, we identify to $\R\times (0,+\infty)$. Suppose that $\overline{g}$ has order $q$, that is, $q$ is the smaller positive integer such that $\overline{g}^q=Id$. We denote $\mathrm{ord}(g)$ = q.
Then there exists an integer $p$ such that, for every $\widetilde{x}\in \widetilde{\A}$, $$\widetilde{\overline{g}}^q(\widetilde{x})= \widetilde{x}+(p,0).$$ It is not hard to prove that $p$ and $q$ are coprime. Besides, every point is recurrent for $\widetilde{\overline{g}}$ and has a rotation number around $\mathbf{0}$ equal to $p/q$. Because of the invariance under conjugacy by $\rho_{\mathbf{0}}(\widetilde{\overline{g}})$ (Property (1) above), this number does not depend on the choice of the chart $\phi$. Hence, by Property (2) above, we can associate  to our periodic homeomorphism $g$ a unique ``local rotation number'' around $z$ defined as 
$$ \lrho(g)= \frac{p}{q} (\text{mod } 1) \in \T{1}.$$
Clearly, if $\lrho(g)=0$, then $g$ is the identity.

Given $z\in \s^2$, we will denote $\homeog{+}{}{\s^2;z}$ the group of all homeomorphisms of $\s^2$ that preserve the orientation and fix $z$.
By the discussion above, the ``local rotation number map'' is well defined for a periodic subgroup of $\homeog{+}{}{\s^2;z
}$. We have the following properties.

\begin{lema}\label{lemapropertiesrotationset}
Let $g$ be an element of $\homeog{+}{}{\s^2;z}$.
  \begin{enumerate}
    \item The local rotation set around $z$ is invariant under (local) oriented topological conjugacy. More precisely, if $\varphi$ belongs to $\homeog{+}{}{\s^2;z}$, then $$    \lrho( \varphi^{-1}g \varphi)=\lrho(g).$$
    \item For every $q\in\Z$, we have  $ \lrho(g^q)=q\lrho(g)$.
  \end{enumerate}
\end{lema}

The first nontrivial observation concerning the local rotation set is the following.

\begin{prop}\label{local rotation map homomorphism}
  Let $G_0$ be a periodic subgroup of $\homeog{+}{}{\s^2;z}$. The local rotation number map defined on $G_0$ is a group homomorphism into $\T{1}$ if and only if $G_0$ is abelian.
\end{prop}
\begin{proof}
   Let $f$ and $g$ be two elements in $G_0$. We recall that $[f,g]:=fgf^{-1}g^{-1}$ denotes the commutator of $f$ and $g$. If $\lrho$ is a group homomorphism, then $\lrho([f,g])$ is nul, which implies that $[f,g]=Id$. As $f$ and $g$ are arbitrary, $G_0$ is abelian.

  Conversely, assume that $G_0$ is Abelian, and let $f$ and $g$ be two elements in $G_0$. Consider a chart $\phi$ centered at $z$, and let
  $\overline{f} := \phi f\phi^{-1}$ and $\overline{g} := \phi g\phi^{-1}$ be the conjugate homeomorphisms. Both $\overline{f}$ and $\overline{g}$
  are periodic homeomorphisms of the plane $\R^2$ that preserve the orientation and fix the vector $\phi(z)=\mathbf{0}\in \R^2$. Since
  $\overline{f}$ and $\overline{g}$ commute, we can consider commuting lifts $\widetilde{\overline{f}}$ and $\widetilde{\overline{g}}$ of $\overline{f}$ and $\overline{g}$,
  respectively.  Suppose that $\rho_{\mathbf{0}}(\widetilde{\overline{f}})=\frac{p'}{q'}$ and $\rho_{\mathbf{0}}(\widetilde{\overline{g}})=\frac{p}{q}$. Then for every $\widetilde{x}$ and $\widetilde{x}'$ in $\widetilde{\A}$, we have
  $$ \widetilde{\overline{g}}^q(\widetilde{x})= \widetilde{x}+(p,0) \quad \text{and} \quad \widetilde{\overline{f}}^{q'}(\widetilde{x}')= \widetilde{x}'+(p',0).            $$
  Thus,
  \begin{align*}
    (\widetilde{\overline{f}}\widetilde{\overline{g}})^{q'q}(\widetilde{x})&= \widetilde{\overline{f}}^{q'q}\left(\widetilde{\overline{g}}^{q'q}(\widetilde{x})\right)\\
                                                     &= \widetilde{\overline{g}}^{q'q}(\widetilde{x})+(qp',0)\\
                                                     &=\widetilde{x}+(q'p,0)+(qp',0).
  \end{align*}
  Therefore $$ \lrho(fg)= \frac{q'p+qp'}{q'q}=  \frac{p}{q}+\frac{p'}{q'}=  \lrho(f)+ \lrho(g).   $$
  This shows that $\lrho$ is a group homomorphism.
\end{proof}

\subsection{Consequences for Abelian, periodic subgroups of ori-\\entation-preserving homeomorphisms of $\s^2$ that fix a point}

From Proposition \ref{local rotation map homomorphism}, we know that the ``local rotation number map'' is an injective group homomorphism for Abelian, periodic subgroups of $\homeog{+}{}{\s^2;z}$. We deduce the following result.

\begin{lema}\label{lemagruposabelianos1}
   Let $A$ be an Abelian, periodic subgroup of $\homeog{+}{}{\s^2;z}$. Let $a$ and $b$ be two elements of $A$ with $\ord(a)=\ord(b)$. Then there exists an integer $i\in \{1,\ldots, \ord(a)-1\}$ such that $b=a^{i}$. In particular there exists at most one element of order $2$ in $A$.
\end{lema}
\begin{proof}
  Since $\ord(a)=\ord(b)$, there exists an integer $i\in \{1,\ldots, \ord(a)-1\}$ such that $\lrho(b)=\lrho(a^{i})$. As the local rotation number map, $\lrho$, is injective on $A$, we deduce that $b=a^{i}$.
\end{proof}

\begin{lema}\label{lemagruposabelianos2}
  If $A$ is a finite, Abelian subgroup of $\homeog{+}{}{\s^2;z}$, then $A$ is cyclic.
\end{lema}
\begin{proof}
  Let $b$ be an element of $A$ with maximal order, and let $B$ be the subgroup generated by $b$. Let $a$ be an element of $A$. If $\ord(a)$ does not divide $\ord(b)$, then $ab$ has order the lowest commom multiple of $\ord(a)$ and $\ord(b)$. This contradicts the maximality of $\ord(b)$. Therefore, $\ord(a)$ divides $\ord(b)$. Hence, there is an integer $m\geq 1$ such that $\ord(b)=m\ord(a)$, and so $\ord(b^{m})=\ord(a)$. By the previous lemma, there exists an integer $i\in \{1,\ldots, \ord(a)-1\}$ such that $a=b^{mi}$. It follows that $a$ belongs to $B$. As $a$ is arbitrary, we deduce that $A=B$.
\end{proof}

\section{Burnside problem for $2$-groups of homeomorphisms of $\s^2$: particular cases}

In order to prove Theorem A, it is enough to prove that every finitely-generated $2$-group $G$ of \textit{orientation-preserving} homeomorphisms of the $2$-dimensional sphere is finite. This is the purpose of the following two subsections. Our proof consists in  first considering the case where $G$ has a global fixed point and latter the case  where $G$ has a finite orbit of cardinality $2$. Finally, we settle the general case, and for this we prove that the group $G$ contains only a finite number of involutions.

\subsection{The case where the group has a global fixed point}

In this section we will prove Theorem B, that is, every finitely-generated $2$-group of orientation-preserving homeomorphisms of the $2$-dimensional sphere that has a global fixed point is finite and cyclic. The idea of the proof is as follows: Notice that a (nontrivial) $2$-group $G_0$ always contains involutions, that is, elements of order $2$. The key step consists in proving that, in our case, there is a unique involution in $G_0$. This implies that such an involution must belong to the center of $G_0$, that is, it commutes with each element of $G_0$. Since we are assuming that $G_0$ is a $2$-group, we can deduce that $G_0$ is abelian using the following property (see Proposition \ref{propocle} below): ``If $f$ and $g^2$ in $G_0$ commute, then $f$ and $g$ commute''. Finally, using that $G_0$ is finitely generated, we can conclude that $G_0$ is finite, and hence
 cyclic by Lemma \ref{lemagruposabelianos2}.

 We start with a lemma that follows from classical properties of the local rotation set around a fixed point (Lemma \ref{lemapropertiesrotationset}).

\begin{lema}\label{lemagorder2}
  Let $g$ be a finite-order element in $\homeog{+}{}{\s^2;z}$. Suppose that $g$ is conjugate (by an element in $\homeog{+}{}{\s^2;z}$) to its inverse. Then $g^2=Id$.
\end{lema}
\begin{proof}
 By hypothesis, there exists $\varphi \in \homeog{+}{}{\s^2;z}$ such that $\varphi^{-1} g \varphi = g^{-1}$. By Lemma \ref{lemapropertiesrotationset},  $$\lrho(g)= \lrho(g^{-1})= - \lrho(g).$$ This implies that $0=2 \lrho(g)= \lrho(g^2)$. Since $g$ has finite order, it must satisfy $g^2=Id$.
\end{proof}

\begin{prop}
    Let $G_0$ be a periodic subgroup of $\homeog{+}{}{\s^2;z}$. Let $\sigma$ and $\sigma'$ be two elements of order $2$ in $G_0$. Then $\sigma$ and $\sigma'$ commute.
\end{prop}
\begin{proof}
  We know that, in any group, $\sigma\sigma'$ is conjugate (by $\sigma$) to $\sigma'\sigma$.
  Indeed, 
  $$ \sigma'\sigma= (\sigma^{-1} \sigma )\sigma' \sigma = \sigma^{-1} (\sigma \sigma' )\sigma.$$
  Since $\sigma$ and $\sigma'$ have order 2, we have that $ (\sigma\sigma')^{-1}= \sigma'\sigma$. By Lemma \ref{lemagorder2}, it follows that $\sigma \sigma'$ have order $2$. Since $\sigma$, $\sigma'$, and $\sigma \sigma'$ have order $2$, we deduce (using an argument due to Burnside) that $$  Id=(\sigma \sigma')^2= \sigma \sigma' \sigma \sigma' = \sigma \sigma' \sigma ^{-1} \sigma'^{-1} :=[\sigma,\sigma'].$$ This implies that $\sigma$ and $\sigma'$ commute.
\end{proof}

We deduce the following properties.

\begin{prop}\label{propuniqueelementofordertwo}
  Let $G_0$ be a periodic subgroup of $\homeog{+}{}{\s^2;z}$. Then $G_0$ has at most one element of order $2$.
\end{prop}
\begin{proof}
  Suppose that $\sigma$ and $\sigma'$ are two elements of order 2 in $G_0$. By the previous proposition, the group generated by $\sigma$ and $\sigma'$ is an Abelian, periodic subgroup of $\homeog{+}{}{\s^2;z}$. By Lemma \ref{lemagruposabelianos1}, we deduce that $\sigma=\sigma'$.
\end{proof}

\begin{coro}\label{sigmaincenterof G}
  Let $G_0$ be a periodic subgroup of $\homeog{+}{}{\s^2;z}$. If $\sigma\in G_0$ has order $2$, then $\sigma$ belongs to the center of $G_0$.
\end{coro}
\begin{proof}
  Let $g$ be an element of $G_0$. Since $g\sigma g^{-1}$ has order $2$, by Proposition \ref{propuniqueelementofordertwo}, we deduce that $g\sigma g^{-1}=\sigma$. This implies that $g$ and $\sigma$ commute.
\end{proof}

\begin{lema}
  Let $G_0$ be a periodic subgroup of $\homeog{+}{}{\s^2;z}$. Let $f$ and $g$ be two elements in $G_0$. Suppose that $f$ and $g^2$ commute. Then for every $i \in \{0,\ldots, \ord(f)-1\}$ the element $[f^{i},g]$, the commutator of $f^{i}$ and $g$, satisfies $[f^{i},g]^2=Id$. Moreover, $[g,f^{i}]^2=Id$
\end{lema}
\begin{proof}
  Since $f$ and $g^2$ commute, we have that $f^{i}$ and $g^2$ commute, that is, $f^{i}g^2=g^2 f^{i}$. Hence
  $$  g^{-1}f^{-i} = g f^{-i}g^{-2}.$$ Consequently,
  $$ [g,f^{i}] = g f^{i}g^{-1}f^{-i}= g f^i(g f^{-i}g^{-2}) = g (f^i g f^{-i} g^{-1})g^{-1} = g [f^i,g] g^{-1}.$$
  Since $ [f^{i},g]^{-1}= [g,f^{i}]$, it follows from Lemma \ref{lemagorder2} that $[f^{i},g]^2=Id$. Finally, notice that $[g,f^{i}]= [f^{i},g]^{-1}$. Hence we deduce that $[g,f^{i}]^2=Id$.
\end{proof}

\begin{prop}\label{propocle}
   Let $G_0$ be a periodic subgroup of $\homeog{+}{}{\s^2;z}$. Let $f$ and $g$ be two elements in $G_0$. If $f$ and $g^2$ commute, then $f$ and $g$ commute.
\end{prop}
\begin{proof}
 By the previous lemma, for $i\in \{0,\ldots, \ord(f)-1\}$, the element $[f^i,g]$ satisfies $[f^i,g]^2=Id$. Hence either $[f^i,g]=Id$ or $[f^i,g]\neq Id$ and $[f^i,g]^2=Id$ (and the same property holds for $[g,f^i]$). Moreover, by Proposition \ref{propuniqueelementofordertwo}, there is a unique element $\sigma$ of order $2$ in $G_0$. Therefore, either $[f^i,g]=Id$ or $[f^i,g]=\sigma$. Suppose by contradiction that $f$ and $g$ do not commute. Then $[f,g]=\sigma$. We claim that
  $$ [f^{i},g]=
  \begin{cases}
     \sigma, \text{ if $i$ is odd,}\\
     Id, \text{ if $i$ is even}.
  \end{cases}  $$
  Indeed, if there exists $i$ such that $[f^{i+1},g]=[f^i,g]$, then this easily yields that $f$ and $g$ commute, which is a contradiction.

  Since $\sigma$ has order $2$, we have that $[f,g]^2=Id$ and so $g[f,g]^2=g $, that is,
  \begin{align*}
     g & = g(fgf^{-1}g^{-1})(fgf^{-1}g^{-1}) \\
     & =  (gf)(gf^{-1}g^{-1}f)g (f^{-1}g^{-1}) \\
     &  = (gf)[g,f^{-1}] g (f^{-1}g^{-1}) .
  \end{align*}

  Applying the local rotation map and using its invariance under (local) topological conjugacy, we obtain $$ \lrho ([g,f^{-1}] g )= \lrho(g).  $$
  We know that $\sigma=[g,f^{-1}]=[g,f^{\ord(f)-1}]$ commutes with $g$ (Corollary \ref{sigmaincenterof G}), and that the local rotation map restricted to the group generated by $\sigma$ and $g$ is a group homomorphism (Proposition \ref{local rotation map homomorphism}).  Thus $$ \lrho(\sigma) +\lrho(g )= \lrho(g).  $$ Therefore, $\lrho(\sigma)=0$, and so $\sigma=Id$. This contradiction proves that $f$ and $g$ commute.
\end{proof}

\begin{proof}[\textbf{End of the proof of Theorem B}]
  Let $G_0$ be a finite-generated $2$-group contained in $\homeog{+}{}{\s^2;z}$. Let $f$ and $g$ be two elements in $G_0$. Since $G_0$ is a $2$-group, $g$ has order $2^{p+1}$ for a certain integer $p\geq 0$, and then $g^{2^p}$ has order $2$. It follows from Corollary \ref{sigmaincenterof G} that $g^{2^p}$ and $f$ commute. Applying the previous proposition, we obtain that $g^{2^{p-1}}$ and $f$ commute. Iterating this argument, we get that $f$ and $g$ commute. Therefore, $G_0$ is an Abelian, finitely-generated group, and so it is finite. Finally, we deduce from Lemma \ref{lemagruposabelianos2} that $G_0$ is cyclic.
\end{proof}

\subsection{The case where the group has an orbit of cardinality~$2$} \label{subsection proof of theorem c}

In this section, we prove Theorem C, that is, every finitely-generated $2$-group $G$ of orientation-preserving homeomorphisms of the $2$-dimensional sphere which has an orbit of cardinality $2$ is finite. More precisely, it is either a cyclic or a dihedral group. Let $z$ be a point with $G$-orbit of cardinality $2$. We write $\mathcal{O}_G(z)=\{z,z'\}$. We will consider the subgroup $G_0$ of the homeomorphisms that fix both $z$ and $z'$.

\begin{lema}\label{lema1 proof theorem c}
  The group $G_0$ is an index-$2$, normal subgroup of $G$. In particular, $G_0$ is a finitely-generated $2$-group contained in $\homeog{+}{}{\s^2;z}$.
\end{lema}
\begin{proof}
  It is easy to check that $G_0$ is normal in $G$. Moreover, notice that if $\sigma$ and $\sigma'$ are in $G\setminus G_0$, then $\sigma\sigma'$ is in $G_0$. Hence, $G_0$ has index $2$ in $G$. Moreover, Schreier's lemma states that any finite-index subgroup in a finitely-generated group is finitely generated. Hence, as $G$ is a finitely-generated $2$-group, we deduce that $G_0$ is a finitely-generated $2$-group contained in $\homeog{+}{}{\s^2;z}$.
\end{proof}

\begin{lema}
  Every $g\in G\setminus G_0$ has order $2$.
\end{lema}
\begin{proof}
  If $g\in G\setminus G_0$, then $g(z)=z'$ and $g(z')=z$. We deduce that $g^2(z)=z$. As the local rotation number of $g$ is a singleton, we deduce that $g$ has order $2$ (see Proposition \ref{kerejartolemma}).
\end{proof}

\begin{proof}[\textbf{End of the proof of Theorem C}]
   By Theorem B, we know that $G_0$ is a finite cyclic group. If $G=G_0$, then $G$ is finite and cyclic. Otherwise, let $g_0$ in $G$ be a generator of $G_0$, and let $g\in G\setminus G_0$. Consider $\Gamma$ the subgroup de $G$ generated by $g_0$ and $g$. We claim that $\Gamma=G$. Indeed, if $g'$ is any element in $G\setminus G_0$, then $gg'\in G_0$, hence $g'\in \Gamma$. Moreover, as $gg_0$ and $g$ do not belong to $G_0$, we have that $gg_0$ and $g$ have order $2$ (by the previous lemma). Hence we have $gg_0 gg_0=Id=g^{2}$, which yields $gg_0g^{-1}=g_0^{-1}$. It follows that $G$ is a dihedral group.
\end{proof}

\section{Burnside problem for $2$-groups of homeomorphisms of the $2$-dimensional sphere}

In this section, we prove Theorem A, that is, every finitely-generated $2$-group $G$ of homeomorphisms of the $2$-dimensional sphere for which there is a uniform bound for the order of group elements is finite. Remind that a nontrivial $2$-group always contains involutions, that is, elements of order $2$. Let us denote $\inv(G):=\{g \in G\setminus {Id}: g^2=Id \}$,
and let $Z(\sigma)$ be the centralizer of $\sigma$ in $G$, that is, $Z(\sigma)=\{g\in G: g\sigma=\sigma g\}.$ In order to prove Theorem A, we start by proving, using Theorem C, that for every $\nu \in \inv(G)$, the set $Z(\nu)\cap \inv(G)$ is finite (along the proof of Theorem A, this is only part where we use the existence of a uniform bound for the order of group elements). Then we will prove that the set $\inv(G)$ is finite. Since each $g\in G\setminus\{Id\}$ has exactly two fixed points (by Proposition \ref{kerejartolemma}), we obtain that the union of fixed points of the involutions is also finite. Moreover, this set is non-empty, $G$-invariant, and has cardinality larger than or equal to $2$. In the case where the cardinality of $F$ is $2$, we deduce that $G$ is finite by Theorem C. For $\sigma \in G$, let us denote $\fix(\sigma)$ the set of all fixed points of $\sigma$.

\begin{prop}\label{prop involutions finite}
  Let $G$ be a finitely-generated $2$-group of orientation-preserving homeomorphisms of $\s^2$. Suppose that $G$ has uniformly bounded order. Then the following assertions hold:
  \begin{itemize}
    \item[(1)] If $\nu \in \inv(G)$, then the set $Z(\nu)\cap \inv(G)$ is finite.
    \item[(2)] The set $\inv(G)$ is finite.
  \end{itemize}
\end{prop}
  \begin{proof}
Let us prove (1).  Suppose that there is an infinite sequence $\nu, \nu_1,\ldots, \nu_n,\ldots$ contained in $Z(\nu)\cap \inv(G)$. Fix an interger $n\geq 1$. The group $G_n$ generated by $\nu, \nu_1,\ldots, \nu_n$ is finitely generated, periodic, and fixes the set of fixed points of $\nu$ (because each $\nu_i$ commutes with $\nu$). Then, by Theorem C, the group $G_n$ is finite and either cyclic or dihedral. Moreover, $$ \{Id\} \subset G_0 \subset \ldots \subset G_n \ldots  .$$ Since we are assuming that elements in $G$ have  uniformly bounded order, this sequence must stabilize at some integer $n_0$. That is, for every integer $n\geq n_0$, one has $G_n=G_{n_0}$. This proves that $Z(\nu)\cap \inv(G)$ is finite.\\

Let us prove (2). Suppose that there exists an infinite sequence of involutions $\sigma_1, \sigma_2,\ldots, \sigma_n,\ldots$ contained in $G$. For every integer $n$, the group generated by $\sigma_1$ and $\sigma_n$ is either cyclic or dihedral. Therefore, it contains an involution $\nu_n$ that commutes with $\sigma_1$ and $\sigma_n$ ($\nu_n=\sigma_1$ in the cyclic case, and $\nu_n=(\sigma_1\sigma_n)^{\ord(\sigma_1\sigma_n) /2}$ in the dihedral case). Since $Z(\sigma_1)\cap \inv(G)$ is finite (by item (1)), we can suppose (eventually passing to an subsequence of $(\sigma_n)_{n\in\N}$) that $\nu_n=\nu$ for every integer $n$. This implies that $\nu$ commutes with all $\sigma_n$, and hence the sequence $\{\sigma_n\}_{n\in \N}$ is contained in $Z(\nu)\cap \inv(G)$. But this last set is finite by item (1). This contradiction proves that the set $\inv(G)$ is finite.
\end{proof}

\begin{proof}[\textbf{End of the proof of Theorem A}]
Assume that $G$ is nontrivial. Applying Proposition \ref{prop involutions finite}, we obtain that the set $\inv(G)$ is finite. Since each $g\in G\setminus\{Id\}$ has exactly two fixed points (by Proposition \ref{kerejartolemma}), we obtain that the set $$F:= \bigcup_{\sigma\in \inv(G)  } \fix(\sigma) $$ is also finite. As $\fix(g\sigma g^{-1})=g(\fix(\sigma))$ and $g\sigma g^{-1}$ is an involution, the set $F$ is non-empty, $G$-invariant, and has cardinality larger than or equal to $2$. In the case where the cardinality of $F$ is $2$, we deduce that $G$ is finite by Theorem C. The case where $F$ contains more than $2$ points is settled by the next result.
\end{proof}

\begin{lema}[Corollary 2 from \cite{gl2}]
  Let $G$ be a finitely-generated periodic group of orientation-preserving homeomorphisms of $\s^2$. Suppose that $G$ has a finite orbit of cardinality at least $3$. Then $G$ is finite.
\end{lema}

\section{Burnside problem for area-preserving homeomorphisms of the $2$-dimensional sphere}

In this section, we prove Theorem D, that is, every finitely-generated, periodic group of area-preserving homeomorphisms of the $2$-dimensional sphere having uniform
order is finite. As in the case of a $2$-group, we start by proving Theorem $E$ (which is the analog of Theorem C in the area-preserving setting). Then using Theorem E we deduce
that Proposition \ref{prop involutions finite} holds in the area-preserving case (in the case where the set $\inv(G)$ is nonempty). We then finish the proof of Theorem D as that of
Theorem A. In order to prove Theorem E, we first introduce the rotation set for a homeomorphism of the open annulus.

\subsection{Rotation set for a homeomorphism of the open annulus}

We let $\A=\T{1}\times \R$ be the open annulus and $\widetilde{\A}:=\R\times \R$ its universal covering. We denote
$\fonc{\widetilde{\pi}}{\widetilde{\A}}{\A}$ the corresponding universal covering map, and $\fonc{p_1}{\R\times \R}{\R}$ the projection on the first coordinate.
By the two-point compactification, one can identify $\A$ to the punctured sphere $\s^2\setminus \{N,S\}$, where $N$ and $S$ are two distinct points of $\s^2$
(the north and south poles). The Lebesgue measure on $\s^2$ induces a probability measure on $\A$, that we still call the Lebesgue measure and denote by $\leb$.

Let $\overline{g}$ be a homeomorphism of $\A$ that is isotopic to the identity, and let $\widetilde{\overline{g}}$ be a lift of $\overline{g}$ to $\widetilde{\A}$. Following \cite{lecalvez}, we say that the {\em rotation number} of a $\overline{g}$-recurrent point $x\in \A$ under $\widetilde{\overline{g}}$ is well-defined and equal to $\rho(\widetilde{\overline{g}},x)\in \R\cup \{-\infty\}\cup \{+\infty\}$ if for every sequence of integers $(n_k)_{k\in \N}$ which goes to $+\infty$ such that $ (\overline{g}^{n_k}(x))_{k\in \N}$ converges to $x$,  the sequence $(\rho_{n_k}(\widetilde{\overline{g}},x))_{k\in \N}$ defined as
$$ \rho_{n_k}(\widetilde{\overline{g}},x):=\frac{1}{n_k}(p_1(\widetilde{\overline{g}}^{n_k}(\widetilde{x}))-p_1(\widetilde{x})),$$
where $\widetilde{x}$ is a point in $\widetilde{\pi}^{-1}(x)$, converges to  $\rho(\widetilde{\overline{g}},x)$. Again, this definition does not depend on the choice of $\widetilde{x}\in \widetilde{\pi}^{-1}(x)$.

Assume that $\overline{g}$ preserves a probability measure $\mu$ on $\A$. We say that the {\em rotation number} of $\tilde{\overline{g}}$ (with respect to $\mu$)
is well defined and equal to $\rho(\widetilde{\overline{g}},\mu)$ if
\begin{enumerate}
  \item $\mu$-almost every point $x\in \A$ has a rotation number $\rho(\widetilde{\overline{g}},x)$,
  \item the function $x \mapsto \rho(\widetilde{\overline{g}},x)$ is $\mu$-integrable, and
\end{enumerate}
$$ \rho(\widetilde{\overline{g}},\mu):= \int_{\A} \rho(\widetilde{\overline{g}},x)\,d\mu.$$
Notice that, by Birkhoff Ergodic Theorem, we have
$$ \rho(\widetilde{\overline{g}},\mu):= \int_{\A} \rho_1(\widetilde{\overline{g}},x)\,d\mu,$$
where $\rho_1(\widetilde{\overline{g}},x)=  p_1(\widetilde{\overline{g}}(\widetilde{x}))-p_1(\widetilde{x})$, with $\widetilde{x}\in \widetilde{\pi}^{-1}(x)$.

\subsection{Rotation set for periodic homeomorphisms of $\s^2$}

In our setting, let $g$ be a periodic, orientation-preserving homeomorphism $\s^2$ that preserves the Lebesgue measure. We know that if $g$ is nontrivial, then
it fixes two distinct points $N$ and $S$ of $\s^2$. As in the local case, we can associate  to our periodic homeomorphism $g$ a unique ``rotation number'' on the open annulus $\A_{N,S}:=\s^2\setminus \{N,S\}$  defined as 
$$ \rho_{\A_{N,S}}(g):= \int_{\A_{N,S}} \rho_1(g,x)\,d\leb \in \T{1}.$$
Notice that if $\rho_{\A_{N,S}}(g)=0$, then $g$ is the identity.

Given $N$ and $S$ two distinct points of $\s^2$, we will denote $\homeog{}{0}{\A_{N,S}}$ the group of all homeomorphisms of $\s^2$ that preserve the orientation and fix both $N$ and $S$. As in the local case we have the following result.

\begin{prop}\label{rotation map homomorphism}
  Let $G_0$ be a periodic subgroup of $\homeog{}{0}{\A_{N,S}}$. The rotation number map defined on $G_0$ is a group homomorphism into $\T{1}$ if and only if $G_0$ is abelian.
\end{prop}

\subsection{Proof of Theorems D and E}

We start by proving Theorem E.

\begin{proof}[\textbf{Proof of Theorem E}]
   Let $G$ be a finitely-generated periodic group of area-preserving homeomorphisms of the $2$-dimensional sphere. Let $z$ be a point with $G$-orbit of cardinality $2$. We write $\mathcal{O}_G(z)=\{z,z'\}$. We consider the subgroup $G_0$ of the homeomorphisms that fix both $z$ and $z'$. By Lemma \ref{lema1 proof theorem c}, the group $G_0$ is an index-2, normal subgroup of $G$. In particular $G_0$ is a finitely-generated periodic group contained in $\homeog{}{0}{\A_{z,z'}}$ all of whose elements preserve the Lebesgue measure. Since
   the  rotation number is a group homomorphism in the are-preserving case (see Lemma \ref{obvio} below), we can invoke to an analog of
Proposition \ref{rotation map homomorphism} to conclude that $G_0$ is Abelian.

\begin{lema}\label{obvio}
Let $G_0$ be subgroup of $\homeog{}{0}{\A_{z,z'}}$. Suppose that each element of $G_0$ preserves the Lebesgue measure. Then the rotation map is a group homomorphism.
\end{lema}
\begin{proof}
Let $f$ and $g$ be two elements of $G_0$. We have that
\begin{align*}
  \rho_{\A_{z,z'}}(fg) & = \int_{\A_{z,z'}} \rho_1(fg,x)\,d\leb (x)   \\
   & = \int_{\A_{z,z'}} \rho_1(f,g(x))\,d\leb(x) + \int_{\A_{z,z'}} \rho_1(g,x)\,d\leb (x) \\
   & = \int_{\A_{z,z'}} \rho_1(f,y)\,d\leb(y) + \int_{\A_{z,z'}} \rho_1(g,x)\,d\leb (x) \\
   & =\rho_{\A_{z,z'}}(f)+ \rho_{\A_{z,z'}}(g).
\end{align*}
This shows that $\rho_{\A_{z,z'}}$ is a group homomorphism.
\end{proof}

Since $G_0$ is finitely generated, periodic, and Abelian, we deduce that it is finite. Moreover by an analog of Lemma \ref{lemagruposabelianos2}
(using the rotation number instead the local rotation set), we deduce that $G_0$ is cyclic. The proof finishes as the proof of Theorem C.
\end{proof}

Now we can prove Theorem D.

\begin{proof}[\textbf{Proof of Theorem D}]
The proof is a straightforward adaptation of the proof of Theorem A. Let $G$ be a finitely-generated periodic group of orientation-preserving homeomorphisms of $\s^2$. Suppose that each element of $G$ preserves the Lebesgue measure, that $G$ has at least an element of even order, and that $G$ has uniformly bounded order. Let us denote $\inv(G):=\{g \in G\setminus {Id}: g^2=Id \}$. Notice that $G$ always contains involutions. Indeed, if $g^{2p}=Id$ for some integer $p$, then $g^p \in \inv(G)$. Applying Proposition \ref{prop involutions finite} (using Theorem E instead of Theorem C), we obtain that the set $\inv(G)$ is finite. The proof follows as the proof of Theorem A.
\end{proof}

\bibliographystyle{alpha}
\bibliography{bibliografie2}

\noindent Jonathan Conejeros\\
Departamento de Matem\'atica y Ciencia de la Computaci\'on\\
Universidad de Santiago de Chile\\
Avenida Libertador Bernardo O''Higgins n°3363,\\
Estaci\'on Central, Santiago, Chile.\\
e-mail: jonathan.conejeros@usach.cl

\end{document}